\documentclass[11pt]{amsart}

\usepackage[section]{placeins}
\usepackage{url}
\usepackage{microtype}
\usepackage{booktabs, comment}
\usepackage[french, english]{babel}
\usepackage{amsmath,amssymb,amsthm, comment, mathrsfs, url}
\usepackage{array}
\usepackage[colorlinks]{hyperref}
\hypersetup{citecolor=blue}
\usepackage{amscd}
\usepackage{xypic}
\usepackage{todonotes}
\usepackage[active]{srcltx}

\theoremstyle{plain}
\newtheorem{thm}{Theorem}[section]

\newtheorem{conj}[thm]{Conjecture}
\newtheorem{rem}[thm]{Remark}

\newtheorem{lem}[thm]{Lemma}

\newcommand{\field}[1]{\mathbb{#1}}
\newcommand{\Qbar}{{\kern.1ex\overline{\kern-.1ex\Q\kern-.1ex}\kern.1ex}}
\newcommand{\CO}{\mathcal{O}}
\newcommand{\Q}{\field{Q}}
\newcommand{\C}{\field{C}}
\newcommand{\R}{\field{R}}
\newcommand{\bfR}{\mathbf{R}}

\newcommand{\Proj}{\mathbf{P}}

\newcommand{\Z}{\field{Z}}
\newcommand{\A}{\field{A}}
\newcommand{\F}{\field{F}}
\newcommand{\G}{\field{G}}

\newcommand{\ov}{\overline}
\newcommand{\p}{\mathfrak{p}}

\newcommand{\gN}{\mathfrak{N}}

\newcommand{\gc}{\mathfrak{c}}

\DeclareMathOperator{\Gal}{Gal}
\DeclareMathOperator{\GL}{GL}

\DeclareMathOperator{\GSp}{GSp}
\DeclareMathOperator{\Sp}{Sp}
\DeclareMathOperator{\GSO}{GSO}
\DeclareMathOperator{\GO}{GO}
\DeclareMathOperator{\End}{End}

\DeclareMathOperator{\im}{Im}
\DeclareMathOperator{\trace}{{\mathrm{Tr}}}
\DeclareMathOperator{\Frob}{Frob}

\title[Lifts of Bianchi Modular Forms]{Theta Lifts of Bianchi Modular Forms and Applications to Paramodularity}

\author{Tobias Berger}
\address{School of Mathematics and Statistics, University of Sheffield, Hicks Building, Hounsfield Road, Sheffield, S3 7RH, UK }
\email{t.t.berger@sheffield.ac.uk}

\author{Lassina Demb\'el\'e}
\address{Mathematics Institute, Zeeman Building, University of Warwick, Coventry, CV4 7AL, UK}
\email{L.Dembele@warwick.ac.uk}

\author{Ariel Pacetti}
\address{Departamento de Matem\'atica, Universidad de Buenos Aires, Ciudad Universitaria
and IMAS, Buenos Aires, Argentina}
\email{apacetti@dm.uba.ar}

\author{Haluk \c{S}eng\"un}
\address{Mathematics Institute, Zeeman Building, University of Warwick, Coventry, CV4 7AL, UK}
\email{M.H.Sengun@warwick.ac.uk}

\begin{document}
\date{\today}

\thanks{The authors' research is supported by the following grants: EPSRC First Grant EP/K01174X/1 (Berger),
EPSRC CAF Fellowship EP/J002658/1 (Demb\'el\'e), CONICET PIP 2010-2012 and FonCyT BID-PICT 2010-0681 (Pacetti), Marie Curie Intra-European Fellowship (\c{S}eng\"un).}

\maketitle

\begin{abstract}
  We explain how the work of Johnson-Leung and Roberts \cite{JLR} on
  lifting Hilbert modular forms for real quadratic fields to Siegel
  modular forms can be adapted to imaginary quadratic fields. For this
  we use archimedean results from~\cite{HST} and replace the global
  arguments of~\cite{Ro01} by the non-vanishing result of
  Takeda~\cite{Takeda2009}. As an application of our lifting result,
  we exhibit an abelian surface $B$ defined over $\Q$, which is not
  restriction of scalars of an elliptic curve and satisfies the
  Brumer-Kramer Paramodularity Conjecture~\cite{brumer-kramer}.
\end{abstract}

\section{\bf Introduction}
The following is a special case of~\cite[Conjecture~1.4]{brumer-kramer}, known as the Brumer-Kramer
Paramodularity Conjecture. (For definitions and ter\-mi\-no\-lo\-gy, we refer the reader to 
Sections~\ref{sec:bianchi-mod-forms} and~\ref{sec:siegel-mod-forms}.)

\begin{conj}\label{conj:BK} Let $B$ be an abelian surface defined over $\Q$ of conductor $N$ such that $\End_\Q(B) = \Z$. 
Then, there exists a Siegel newform $g$ of genus $2$, weight $2$ and paramodular level $N$ such that $L(B, s) = L(g, s)$. 
Conversely, if $g$ is a Siegel newform of genus $2$, weight $2$ and paramodular level $N$, which is a non-Gritsenko lift 
and whose Hecke eigenvalues are integers, then there exists an abelian surface $B$ defined over $\Q$ such that 
$\End_\Q(B) = \Z$ and $L(g, s) = L(B, s)$. 
\end{conj}
In~\cite{brumer-kramer}, there are examples of abelian surfaces of prime conductor. The first few Euler factors 
for each of them are shown to match those of a paramodular form in~\cite{poor-yuen09}. However, none of those 
surfaces has been proved to be modular. 

Let $K$ be a quadratic field, and $E$ an elliptic curve defined over
$K$ which is a {\it non-base change}.  The surface $B_E
=\mathrm{Res}_{K/\Q}(E)$ is an abelian surface defined over $\Q$ with
$\End_\Q(B_E) = \Z$.  When $K$ is real quadratic, recent work
of~\cite{FHS13} shows that there exists a Hilbert newform $f$ of
weight $2$ such that $L(E, s) = L(f, s)$. The form $f$ can be lifted
to a paramodular Siegel cusp form $g$ of genus $2$ and weight $2$
thanks to work of Johnson-Leung and Roberts~\cite{JLR}. A direct consequence
of the construction of their theta lift is that $L(B_E, s) = L(g,
s)$. In other words, the surface $B_E$ is paramodular. So elliptic
curves that are non-base change over real quadratic fields provide a
large supply of abelian surfaces which satisfy Conjecture~\ref{conj:BK}. 

In fact, the above strategy to produce evidence for Conjecture~\ref{conj:BK} was elaborated by Brumer and Kramer themselves. 
They also speculated that further evidence could be gathered by using abelian surfaces $B$ with trivial endomorphism ring 
over $\Q$ such that $\End_{\Qbar}(B) \supsetneq \Z$ (see the paragraph following the statement of~\cite[Conjecture 1.4]{brumer-kramer}).
In~\cite{dk}, Demb\'el\'e and Kumar provide such numerical evidence. They give explicit examples of paramodular abelian surfaces 
$B$ defined over $\Q$ which become of $\GL_2$-type over some real quadratic fields. Like with the surfaces $B_E$ above, 
the Johnson-Leung-Roberts' lift also plays a crucial r\^ole in their work.

The goal of this paper is twofold. First, we show that one can generalise the construction
in~\cite{JLR} to imaginary quadratic fields. For this one needs to replace the theta correspondence between ${\rm GO}(2,2)$ 
and ${\rm GSp}_4$ by the one between ${\rm GO}(3,1)$ and ${\rm GSp}_4$ used by Harris, Taylor and Soudry in their work 
on Galois representations associated to cusp forms for ${\rm GL}_2$ over imaginary quadratic fields (see \cite{HST, taylor_94}). 
Our analysis is then an adaptation of the one in~\cite{Ro01} and~\cite{JLR} using archimedean results from~\cite{HST} and 
the local-global non-vanishing result of Takeda~\cite{Takeda2009}. Our lift can be seen as a type of 
Yoshida lift, although, strictly speaking, Yoshida~\cite{yoshida80} only considered the groups ${\rm O}(4)$ and ${\rm O}(2,2)$.


Second, we combine our lift with explicit computations of Bianchi mo\-du\-lar forms to
exhibit an abelian surface $B$ defined over $\Q$, which satisfies Conjecture~\ref{conj:BK} but is not
a restriction of scalars of an elliptic curve. Because of the difficulties in constructing such
examples (see Section~\ref{sec:abelian-surface}) and the paucity of modularity results for $\GL_2$ over imaginary
quadratic fields, we had to limit ourselves to one example.
However, it should be clear to the reader that our method, which borrows from~\cite{dk}, 
can be used in principle to generate more cases of the conjecture when further
modularity results become available in this case.

As in the real quadratic case, if $E$ is a modular non-base change elliptic curve  over an imaginary quadratic field, 
then the associated surface $B_E$ is paramodular by our lifting result.  We note that there are several examples of 
such elliptic curves in the literature (see~\cite{DGP} and references therein). 

The paper is organised as follows. In Sections~\ref{sec:bianchi-mod-forms}
and~\ref{sec:siegel-mod-forms}, we recall the necessary background on
Bianchi and Siegel modular forms. In Section~\ref{sec:lifting}, we
prove the existence of our theta lift for any Bianchi modular form of even weight $k$ that is 
non-base change. In Section~\ref{sec:paramodularity}, we give
an example of an abelian surface $B$ defined over $\Q$ which
satisfies Conjecture~\ref{conj:BK}. We prove the
modularity of our surface (over an imaginary quadratic field) using the so-called Faltings-Serre method~\cite{serre}.\\

\noindent
{\bf Acknowledgements.} We would like to thank Abhishek Saha and Neil Dummigan for  useful conversations and comments.
We are also grateful to Abhinav Kumar for kindly providing us with a preliminary version of his joint work with Elkies~\cite{ek}, 
and for many helpful email exchanges.

\section{\bf Background on Bianchi modular forms}\label{sec:bianchi-mod-forms}

\subsection{Bianchi modular forms}
Let $K$ be an imaginary quadratic field, and $\CO_K$ its ring of integers. 
Let $\widehat{\Z}$ be the finite ad\`eles of $\Z$. The ad\`eles of $K$ are given by 
$\A_K =\C \times \A_K^f$, where $\widehat{\CO}_K = \CO_K\otimes\widehat{\Z}$ and $\A^f_K= K \otimes \widehat{\Z}$
are the finite ad\`eles of $\CO_K$ and $K$, respectively.

\medskip
Fix an ideal $\gN \subseteq \CO_K$ and define the compact open subgroup 
\begin{align*}
U_0(\gN):=\left\{\gamma\in\GL_2(\widehat{\CO}_K):\gamma\equiv\begin{pmatrix}*&*\\ 0&*\end{pmatrix} \bmod\gN\right\}.
\end{align*}
Set $\mathcal{K}= \mathcal{K}_\infty \times U_0(\gN)$ where 
$\mathcal{K}_\infty=\C^\times {\rm U}(2)$.

Let $k\ge 2$ be an {\em even} integer and consider the $\GL_2(\C)$-representation 
$V_k = \mathrm{Sym}^{2k-2}(\C^2)$. Let $\chi: K^\times \backslash \A_K^\times/\C^\times \widehat{\CO}_K^\times \to \C^\times$ be an unramified Hecke character with trivial infinity type. For $\gamma = \left(\!\begin{smallmatrix}a&b\\c&d \end{smallmatrix}\!\right) \in {\rm GL}_2(\A_K)$ put  $\chi_{\gN}(\gamma)=\prod_{v \mid \gN} \chi_v(d_v)$. We define the space $S_k(\gN, \chi)$ of {\em Bianchi cusp forms} of level $U_0(\gN)$, central character $\chi$ and weight $k$ to be the set of all functions $f:\ \GL_2(\A_K)  \to V_k$ such that 
\begin{enumerate}
\item $f(\gamma g)= f(g)$ for all $\gamma \in \GL_2(K)$,
\item $f(z g)= \chi(z)f(g)$ for all $z \in \A_K^\times$,
\item $f(gu) = \chi_{\gN}(u) f(g) u_\infty$ for all $u=(u_\infty,u_f) \in {\rm U}(2) \times U_0(\gN)$, 
\item $f$, viewed as a function on $\GL_2(\C)$, is an eigenfunction of the complexification (in $\mathfrak{sl}_2(\C) \otimes_\R \C$) of the Casimir operator of $\mathfrak{sl}_2(\C)$ with eigenvalue $(k^2-2k)/8$,
\item for all $g \in \GL_2(\A_K)$, one has 
$$\int_{K\backslash\A_K} f\left (\begin{pmatrix} 1 & u \\ 0 & 1 \end{pmatrix} g \right ) du = 0.$$
\end{enumerate} 
We invite the reader to see~\cite{ghate99} and \cite{Hida94} for details. In the notations of \cite{ghate99}, we take $\mathbf{n}=(k-2)i+ (k-2)c$, $\mathbf{v}=-\frac{k-2}{2}i-\frac{k-2}{2} c$ in $ \Z[\{i,c\}]$, for $i,c: K \hookrightarrow \C$, to ensure that our forms have trivial action by $\C^\times$. 
The space $S_k(\gN, \chi)$ is endowed with the so-called Petersson inner product.

\subsection{Newforms and $L$-series} 
For each $\p \nmid \gN$, write the double coset decomposition 
$$U_0(\gN)\begin{pmatrix}\varpi_\p&0\\0 &1\end{pmatrix} U_0(\gN) = \coprod_{i} x_i U_0(\gN),$$
where $\varpi_\p$ is a uniformiser at $\p\nmid\gN$. We define the Hecke operator
\begin{align*}
T_\p:\,S_k(\gN, \chi) &\to S_k(\gN, \chi)
\end{align*}
by $$T_\p f(g) := \sum_{i} \chi_{\gN}(x_i)^{-1} f(g x_i), \, g \in \GL_2(\A_K).$$
Similarly, one can define the operator $U_\p$ for $\p \mid \gN$. 

We refer to~\cite[Section~1.6]{U98} for the theory of newforms for Bianchi modular forms.
When $f \in S_k(\gN, \chi)$ is a newform, we have
\begin{align*}
T_\p f &= a_\p f\,\,\text{for all}\,\,\p \nmid \gN;\\
U_\p f &= a_\p f\,\,\text{for all}\,\, \p \mid \gN.
\end{align*}
The $L$-series attached to $f$ is then defined by 
$$L(f, s) := \prod_{\p \mid \gN}\left(1 - \frac{a_\p}{\mathrm{N}(\p)^s}\right)^{-1} 
\prod_{\p \nmid \gN}\left(1 - \frac{a_\p}{\mathrm{N}(\p)^s} + \frac{\chi(\p)}{\mathrm{N}(\p)^{2s-1}}\right)^{-1}.$$
The newform $f$ detemines a unique cuspidal irreducible automorphic representation $\pi$ of $\GL_2(\A_K)$ of level $\gN$, 
which admits a restricted tensor product (see~\cite{flath})
$$\pi = \pi_{\infty}\otimes \bigotimes_{\p< \infty }\pi_{\p},$$
such that $\pi_{\infty}$ has $L$-parameter (see Section \ref{sec:lifting})
\begin{align*}
W_{\C}=\C^\times &\to 
{\rm GL}_2(\C): z \mapsto  \begin{pmatrix} (z/|z|)^{k-1}&0\\0&(z/|z|)^{1-k}\end{pmatrix} 
\end{align*} and the central character of $\pi$ is $\chi$.

Conversely, any cuspidal automorphic representation $\pi$ of $\GL_2(\A_K)$ whose infinity component has such an $L$-parameter corresponds to a newform 
$f$ of weight $k$.
\subsection{Connections with Cohomology} 
It is standard to pass to the cohomological framework when working with Bianchi automorphic forms.
In fact, this is the only approach that is currently suitable for the algorithmic methods that were used to
gather the data in Subsection~\ref{sec:abelian-surface} (see~\cite{cremona,rahm-sengun, yasaki}). 
So we now give a quick review of this framework. 

Let $W_n=\mathrm{Sym}^{n}(\C^2) \otimes_\C \overline{\mathrm{Sym}^{n}(\C^2)}$ where the action of $\GL_2(\C)$ on the
second factor is twisted with complex conjugation. Consider the adelic
locally symmetric space
$$Y_0(\gN) = \GL_2(K)\backslash\GL_2(\A_K) / \mathcal{K}.$$
The cohomology spaces 
$$H^i(Y_0(\gN), \mathcal{W}_n),$$
where $\mathcal{W}_n$ is the local system induced by $W_n$, carry a natural Hecke action on them. 
Let $Y_0(\gN)^{BS}$ denote the Borel-Serre compactification of $Y_0(\gN)$. This is a compact space with boundary $\partial Y_0(\gN)^{BS}$, made of finitely many $2$-tori, 
that is homotopy equivalent to $Y_0(\gN)$. We define the {\em cuspidal cohomology} $H^i_{cusp}(Y_0(\gN), \mathcal{W}_n)$ as the kernel of the restriction map
$$H^i(Y_0(\gN)^{BS}, \mathcal{W}_n) \longrightarrow H^i(\partial Y_0(\gN)^{BS}, \mathcal{W}_n).$$
The Hecke action stabilises the cuspidal cohomology.

By the strong approximation theorem, the determinant map induces a canonical bijection
\begin{align*}\label{eq:strong-approx}
\GL_2(K) \backslash \GL_2(\A_K^f)/ U_0(\gN) \simeq K^\times \backslash (\A_K^f)^\times/\widehat{\CO}_K^\times \simeq \mathrm{Cl}(K),
\end{align*}
where $\mathrm{Cl}(K)$ is the class group of $K$. 
Let $\gc_i$, $1 \le i \le h$, be a complete set of representatives for the classes in $\mathrm{Cl}(K)$.
For each $i$, let $t_i$ be a finite id\`ele which generates $\gc_i$ and set 
$\Gamma_0(\gc_i, \gN) = \GL_2(K) \cap \left(\GL_2(\C) x_i U_0(\gN)x_i^{-1}\right)$
where $x_i = \left(\!\begin{smallmatrix} 1 & 0 \\ 0 & t_i \end{smallmatrix}\!\right)$. Then, we have 
\begin{align*}
\Gamma_0(\gc_i, \gN) &=\left\{\begin{pmatrix} a &b \\ c& d \end{pmatrix} \in \begin{pmatrix}\CO_K & \gc_i^{-1}\\
\gc_i\gN & \CO_K \end{pmatrix}: ad - bc \in \CO_K^{\times}\right\},
\end{align*} 
and 
$$ Y_0(\gN) = \bigsqcup_{j=1}^h \Gamma_0(\gc_j, \gN) \backslash \mathcal{H}_3$$
where $\mathcal{H}_3 \simeq \GL_2(\C) / \mathcal{K}_\infty$ is the hyperbolic 3-space. It follows that the co\-ho\-mo\-lo\-gy of 
the adelic space decomposes as
$$H^i(Y_0(\gN), \mathcal{W}_{n}) = \bigoplus_{j=1}^h H^i(\Gamma_0(\gc_j, \gN) \backslash \mathcal{H}_3, \mathcal{W}_{n}).$$
The cohomology spaces on the right hand side lend themselves very well to explicit machine calculations. We refer 
to~\cite{rahm-sengun, yasaki} for more details on this, and on the Bianchi newforms related to the examples in this paper. 

We recall that the diagonal embedding 
$$K^\times \backslash (\A_K^f)^\times/\widehat{\CO}_K^\times \hookrightarrow \GL_2(K) \backslash \GL_2(\A_K^f)/ U_0(\gN),$$
induces an action of $\mathrm{Cl}(K)$ on the $H^i_{cusp}(Y_0(\gN), \mathcal{W}_{n})$. 
This action is compatible with the Hecke action as it is via the diamond operators. Therefore, these cohomology spaces decompose
accordingly. 

For $k\ge 2$ even, the Generalised Eichler-Shimura Isomorphism (see~\cite{harder}, in our setting \cite[Proposition~3.1]{Hida94}) says that  
\begin{equation} \label{eq:ESH}
H^1_{cusp}(Y_0(\gN), \mathcal{W}_{k-2}) \simeq \bigoplus_{\chi} S_{k}(\gN, \chi) \simeq H^2_{cusp}(Y_0(\gN), \mathcal{W}_{k-2}) 
\end{equation}
as Hecke modules, where $\chi$ runs over all Hecke characters of trivial infinity type that are unramified everywhere (i.e. characters of $\mathrm{Cl}(K)$).

\begin{rem}\label{rem:central-character}\rm
Our Theorem~\ref{thm:main} only applies to the newforms in $S_k(\gN, \mathbf{1})$, i.e. the newforms
whose central character is trivial. We refer to~\cite{rahm-sengun, yasaki} and~\cite[Sections 2.4 and 2.5]{Lingham} 
for details on how the cohomology corresponding to this space can be computed.
\end{rem}

\section{\bf Background on Siegel modular forms}\label{sec:siegel-mod-forms}

 \subsection{Siegel modular forms} 
 Let $$J = \begin{pmatrix}0&\mathbf{1}_2\\ -\mathbf{1}_2&0\end{pmatrix} \in M_4(\Z).$$
 We recall that the symplectic group of genus $2$ is the $\Q$-algebraic group $\GSp_4$
 defined by setting
 $$\GSp_4(R) := \left\{\gamma \in \GL_4(R): \gamma^t J
 \gamma =\nu(\gamma) J,\,\,\nu(\gamma) \in R^\times \right\},$$
 for any $\Q$-algebra $R$. The map $\nu:\,\GSp_4 \to \G_m$ is called the similitude factor,
 and its kernel is the symplectic group $\Sp_4$.  
 
 Let $\GSp_4(\R)^+$ be the subgroup of $\GSp_4(\R)$ which consists of the matrices
 $\gamma$ such that $\nu(\gamma)>0$; also let 
 $$\mathfrak{H}_2:= \left\{Z \in M_2(\C): Z = Z^t\,\text{and}\,\mathrm{Im}(Z)>0\right\}$$
 be the Siegel upper half plane of degree $2$. 
 We recall that $\GSp_4(\R)^+$ acts on $\mathfrak{H}_2$ by 
 $$\gamma\cdot Z = (AZ+B)(CZ + D)^{-1},\,\gamma:=\begin{pmatrix}A & B\\ C& D\end{pmatrix}.$$
 
 Let $k\ge 2$ be an even integer, and set $V_{(k,2)}={\rm Sym}^{k-2}(\C^2) \otimes {\rm det}^2(\C^2)$;
 this is a $\GL_2(\C)$-representation which we denote by $\rho_{(k,2)}$. The group 
 $\GSp_4(\R)^{+}$ acts on the space of functions $F: \mathfrak{H}_2 \to V_{(k,2)}$ 
 by $$(F|_{(k,2)}\gamma)(Z)=\nu(\gamma)^{\frac{k+2}{2}} \rho_{(k,2)}((CZ+D)^{-1}) F(\gamma \cdot Z).$$

 
 We fix a positive integer $N$, and consider the paramodular group
 $$\Gamma^{p}(N) := \begin{pmatrix}\Z & N\Z & \Z & \Z\\ \Z & \Z & \Z & N^{-1}\Z\\ 
 \Z & N\Z & \Z & \Z \\ N\Z & N\Z & N\Z & \Z\end{pmatrix} \cap \Sp_4(\Q).$$
A {\em Siegel modular form} of genus $2$, weight $(k,2)$ and {\em paramodular} level $N$ is a holomorphic function 
$F:\,\mathfrak{H}_2 \to V_{(k,2)}$ such that $F|_{(k,2)}\gamma = F$ for all $\gamma \in \Gamma^p(N)$. 

Let $F$ be a Siegel modular form of genus $2$, weight $(k,2)$ and paramodular level $N$. Then, we see that
$$F(Z + \mathbf{1}_2) = F(Z),\,\,\text{for all}\,\, Z \in \mathfrak{H}_2.$$
By the Koecher principle~\cite{VDG},  
 $F$ admits a Fourier expansion of the form
$$F(Z) = \sum_{Q\ge 0} a_Q e^{2\pi i \mathrm{Tr}(QZ)},$$
where $Q$ runs over all the $2\times 2$ symmetric matrices in $M_2(\Q)$
that are positive semi-definite. We say that $F$ is a {\em cusp form} if, for all $\gamma \in \GSp_4(\Q)$,
we have 
$$(F|_{(k,2)}\gamma)(Z) = \sum_{Q > 0} a_Q^\gamma e^{2\pi i \mathrm{Tr}(QZ)}.$$
We denote by $S_{(k,2)}(\Gamma^p(N))$ the space
of paramodular Siegel cusp forms of level $N$ and weight $(k,2)$. 

\subsection{Newforms and $L$-series}
We consider the double coset decompositions 
\begin{align*}
\Gamma^p(N)\begin{pmatrix} 1 & & & \\ & 1 & & \\ & & p & \\ & & & p\end{pmatrix}\Gamma^p(N) &= \coprod_{i}\Gamma^p(N) h_i;\\
\Gamma^p(N)\begin{pmatrix} p & & & \\ & 1 & & \\ & & p & \\ & & & p^2\end{pmatrix}\Gamma^p(N) &= \coprod_{i}\Gamma^p(N) h_i';
\end{align*}
and following \cite[p. 546]{JLR} define the Hecke operators
\begin{align*}
 T(p) F&= p^{\frac{k-4}{2}}\sum_i F |_{(k,2)}h_i\\
 T(1, p, p, p^2)F&=p^{k-4}\sum_i F|_{(k,2)}h_i'.
 \end{align*}
(As \cite{JLR} note $T(1,p,p,p^2)$ agrees with the classical $T(p^2)$ for $p \nmid N$. We also note that our Hecke 
operators are scaled so as to match the definition in~\cite[p.164]{Arakawa83}.) We refer to \cite[p.547]{JLR} or~\cite{RS07} 
for the definition of the operators $U_p$ for $p\mid N$.


For the paramodular group $\Gamma^p(N)$, the theory of newforms developed in~\cite{RS2, RS07} 
for scalar weights carries over to the vector-valued setting. The old subspace is generated by the images 
of the level-raising maps of \cite{RS2}. One then defines the new subspace 
$S_{(k,2)}^{\rm\tiny new}(\Gamma^p(N))$ to be its or\-tho\-go\-nal complement with respect to the 
Petersson inner product for vector-valued forms given in~\cite{Arakawa83}.

Let $F \in S_{(k,2)}^{\rm\tiny new}(\Gamma^p(N))$ be a newform such that 
for all $p$
$$T(p)F=\lambda_p F,\,\, T(1, p, p, p^2)F=\mu_p F\,\,\text{and}\,\,U_p F =\epsilon_p F$$ for $\lambda_p,\mu_p,\epsilon_p\in \overline{\Q}$. 
By work of~\cite{AsgariSchmidt}, we can associate an automorphic representation 
$\Pi$ to $F$ which admits a restricted tensor product
$$\Pi = \Pi_\infty \otimes \bigotimes_{p<\infty} \Pi_p.$$
The $L$-series attached to $F$ is defined by
$$L(s, F) := \prod_{p} L_p(s, F),$$ 
where the local Euler factors $L_p(s,F)$ are obtained as follows:
 \begin{enumerate}
 \item for ${\rm val}_p(N)=0$, we use the classical Euler factor in \cite[p.173]{Arakawa83}:
 $$L(s,F)^{-1}=1-\lambda_{p}p^{-s}+(p \mu_{p} + p^{k-1}+p^{k-3})p^{-2s}-p^{k-1}\lambda_{p}p^{-3s}+p^{2k-2}p^{-4s};$$
 \item for ${\rm val}_p(N)=1$, we let
 $$L(s,F)^{-1} = 1-(\lambda_{p}+p^{k/2-2} \epsilon_{p})p^{-s}+(p \mu_{p} + p^{k-1})p^{-2s}+ \epsilon_{p} p^{3k/2-2}p^{-3s};$$
 \item for ${\rm val}_p(N)\geq 2$, we let
 $$L_p(s,F)^{-1}=1-\lambda_{p}p^{-s}+(p \mu_{p} + p^{k-1})p^{-2s}.$$
 \end{enumerate}
For ${\rm val}_p(N) \geq 1$ these definitions are motivated by the results of \cite{RS07} and the definitions in \cite{JLR}.

\section{\bf Theta lifts of Bianchi modular forms}\label{sec:lifting}

We briefly recall the definition of $L$-parameters following \cite[Section~4.1]{Saha2013}. 
Let $E$ be a number field, $v$ a place of $E$, and $E_v$ the completion of $E$ at $v$. 
Let $G$ be the group $\GL_2$ or ${\rm GSp}_{4}$. Then, the 
local Langlands correspondence is known for $G$ (see \cite{Knapp, BushnellHenniart, GT2011}). 
It yields a corresponding finite-to-one surjective map
$$L: \Pi(G(E_v)) \to \Phi(G(E_v)),$$
where 
\begin{itemize}
\item $\Pi(G(E_v))$ is the set of isomorphism classes of irreducible admissible representations of $G(E_v)$;
\item $\Phi(G(E_v))$ is the set of $L$-parameters for $G(E_v)$, i.e. the set of isomorphism classes 
of admissible homomorphisms $\phi: W_{E_v}' \to {}^LG^0$, 
where $W_{E_v}'$ is the Weil-Deligne group of $E_v$ and ${}^LG^0$ the dual group of $G$ (which equals $G(\C)$ in our cases).
\end{itemize}
For any irreducible admissible representation $\pi_v$ of $G(E_v)$, we call $L(\pi_v)$ the $L$-parameter of $\pi_v$.  

\begin{thm} \label{thm:main}
Let $K/\Q$ be an imaginary quadratic field of discriminant $D$. 
Let $\CO_K$ be the ring of  integers of $K$, and $\gN \subseteq \CO_K$ an ideal.
Let $\pi$ be a tempered cuspidal irreducible automorphic representation of ${\rm GL}_2(\A_K)$ 
of level $\gN$ and trivial central character such that $\pi_{\infty}$ has L-parameter 
\begin{align*}
W_{\C}=\C^\times &\to 
{\rm GL}_2(\C): z \mapsto  \begin{pmatrix} (z/|z|)^{k-1}&0\\0&(z/|z|)^{1-k}\end{pmatrix}
\end{align*} for some $k \in \Z, k \geq 2$ even.  Assume that $\pi$ is
not Galois-invariant. Then there exists an irreducible cuspidal
representation $\Pi=\bigotimes_v' \Pi_v$ of ${\rm GSp}_4(\A_{\Q})$
with trivial central character such that
\begin{enumerate}
\item $\Pi_v$ is generic for all finite $v$;
\item $\Pi_{\infty}$ is a holomorphic limit of discrete series of Harish-Chandra weight $(k-1,0)$;
\item the following equality of $L$-parameters holds for all places $v$: 
\begin{equation} \label{Lparam} L(\Pi_v)= \bigoplus_{w \mid v} {\rm Ind}_{W'_{K_w}}^{W'_{\Q_v}}L(\pi_w).
\end{equation} 
\end{enumerate}
Consequently, there exists a Siegel newform $F$ of weight $(k,2)$ and paramodular level  $N=D^2 \mathrm{N}_{K/\Q}(\gN)$ 
with Hecke eigenvalues, epsilon factor and (spinor) $L$-function determined explicitly by $\pi$ (and described in the proof below).
\end{thm}

\begin{rem} \rm
The $L$-parameters for ${\rm GSp}_4(\Q_v)$ on the right hand side of (\ref{Lparam}) are defined as follows:
\begin{enumerate}
\item[(i)]
If $v$ is split then as in \cite{JLR} (5), $L(\pi_w) \oplus L(\pi_{\ov{w}})(x) = \left(\begin{smallmatrix} a_1 & & b_1 & \\
& a_2 & & b_2\\
c_1 & & d_1 & \\
 & c_2 & & d_2
\end{smallmatrix}\right)$, where $L(\pi_w)(x) = \left(\begin{smallmatrix} a_1 & b_1 \\ c_1 & d_1\end{smallmatrix} \right)$ and $L(\pi_{\ov{w}})(x) = \left(\begin{smallmatrix}a_2 & b_2 \\ c_2 & d_2\end{smallmatrix} \right)$. 
\item[(ii)] If $v$ is non-split and finite, as in \cite{JLR} (6), let $g_0 \in W'_{K_w}
  \backslash W'_{\Q_v}$ be non-trivial. Then ${\rm
    Ind}_{W'_{K_w}}^{W'_{\Q_v}}L(\pi_w)=\varphi(K_w, \pi_w,
  \mathbf{1})$, where
  \begin{itemize}
  \item If $y \in W'_{K_w}$, $\varphi(K_w, \pi_w, \mathbf{1})(y) = L(\pi_w)(y) \oplus L(\pi_w)(g_0 y g_0^{-1})$.
  \item $\varphi(K_w,\pi_w,\mathbf{1})(g_0) = \left(\begin{smallmatrix} 
& 1 & & \\
a_0 & & b_0 & \\
& & & 1\\
c_0 & & d_0 & \end{smallmatrix}\right)$, for $L(\pi_w)(g_0^2)=\left(\begin{smallmatrix}a_0 & b_0\\ c_0 & d_0\end{smallmatrix} \right)$.
  \end{itemize}

\item[(iii)] If $v=\infty$ then ${\rm Ind}_{W'_{K_w}}^{W'_{\Q_v}}L(\pi_w)=\varphi(K_w, \pi_w, {\rm sgn})$ as defined in \cite{JLR} (6) (definition extended to the archimedean case). This $L$-parameter  $L(\Pi_{\infty}):W'_{\bfR}=\C^\times \cup \C^\times j \to {\rm GSp}_4(\C)$ (where $j^2=-1$ and $jzj^{-1}=\ov{z}$ for $z \in \C^\times$) can also be described explicitely 
as follows: 
\begin{align*}
j &\mapsto \begin{pmatrix} &&&1\\&&1&\\&(-1)^{k-1}&&\\(-1)^{k-1}&&&\end{pmatrix},\\ \\
z &\mapsto \begin{pmatrix} (z/|z|)^{k-1}&&&\\ &  (z/|z|)^{k-1}&&\\&& (z/|z|)^{1-k}&\\&&& (z/|z|)^{1-k} \end{pmatrix},\,\,z \in \C^\times.
\end{align*}
\end{enumerate}

\end{rem}

\begin{rem}\rm
  For odd weights $k$ one still has a lift to a cuspidal automorphic
  representation $\Pi$ of ${\rm GSp}_4$ with central character given
  by the quadratic character $\omega_{K/\Q}$ corresponding to $K$ such
  that conditions (a) and (b) in the theorem are
  satisfied. (\cite[Lemma~12]{HST} also gives a lift with trivial
  central character, but $\Pi_{\infty}$ must then be generic of
  highest weight $(k,1)$.) We also note that one can, in fact, lift any non-Galois-invariant $\pi$ with cyclotomic central character  to an irreducible cuspidal representation of ${\rm GSp}_4(\A_{\Q})$ (as discussed in \cite{HST} and \cite{Takeda2009}). However, the local calculations of
  \cite{JLR} and the paramodular newform theory of \cite{RS07} apply
  only for trivial central character, which is the reason why we
  exclude odd weights in the theorem and impose the condition of a trivial central character.
\end{rem}

\begin{proof}
The strategy of the proof is very similar to that of \cite{JLR}, but replaces the theta correspondence between ${\rm GO}(2,2)$ and ${\rm GSp}_4$ by that between ${\rm GO}(3,1)$ and ${\rm GSp}_4$. 
There are four steps in the construction of the lift, which are outlined as follows:

\begin{enumerate}
\item As in \cite[Section 1]{HST}, $\pi$ gives rise to an automorphic representation $\sigma=(\pi, \mathbf{1})$ for the identity component ${\rm GSO}(3,1)$ of ${\rm GO}(3,1)$.
\item Choosing suitable extensions of the local components of $\sigma$ promotes this to a representation $\hat{\sigma}$ of ${\rm GO}(3,1)$. (In this step we follow \cite{Ro01} rather than \cite{HST} at the non-archimedean places.)
\item As in \cite{HST} (but using the non-vanishing result of \cite{Takeda2009}) we then use the theta correspondence between ${\rm GO}(3,1)$ and ${\rm GSp}_4$ to lift $\hat{\sigma}$ to the automorphic representation $\Pi$ of ${\rm GSp}_4$ described in the theorem.
\item To produce the paramodular Siegel modular form, one takes the automorphic form $\Phi=\bigotimes \Phi_v$, where $\Phi_v \in \Pi_v$ for $v \nmid \infty$ are the paramodular newform vectors defined by Roberts and Schmidt \cite{RS07}. By referring to the classical treatment in \cite{Arakawa83}, we transfer the local non-archimedean  calculations of Hecke eigenvalues, epsilon and $L$-factors of $\Pi$  in \cite{JLR} to those of the corresponding vector-valued Siegel modular form on $\mathfrak{H}_2$.
\end{enumerate}
\medskip

We now give precise details for each of these steps:

(a) Let $X=\{A \in M_{2}(K): \ov{A}^t=A\}$ be the space of $2\times 2$ Hermitian matrices over $K$ with quadratic form given by $-{\rm det}$. By~\cite[Proposition 1]{HST}, $\sigma=(\pi,\mathbf{1})$ defines an irreducible tempered cuspidal automorphic representation of $\GSO(X,\A_{\Q})\cong ({\rm GL}_2(\A_K) \times \A_{\Q}^*)/\{(z {\rm Id}_2, {\rm N}_{K/\Q} z^{-1}), z \in \A_K^*\}$. 

\medskip

Before we come to step (b) we review some details of the theta correspondence between $\GO(X,\A_{\Q})$ and ${\rm GSp}_4(\A_{\Q})$: Following \cite{HST, Ro01, Takeda2009}, we consider the extension of the Weil representation for $\Sp_4 \times \mathrm{O}(X)$ to the group $R=\{(g,h) \in {\rm GSp}_4 \times \GO(X): \nu(g)\nu(h)=1\}$ and denote this representation by $\omega$. (As explained in  \cite[Remark~4.3]{Takeda2009}, there is a difference in the definition of $R$ in~\cite{HST, Ro01}. But, since we are working with trivial central characters, this does not matter here.) 

Let $G$ be a reductive group over $\Q$, and $v$ a place of $\Q$. We denote by ${\rm Irr}(G(\Q_v))$ the collection of equivalence classes of  irreducible smooth admissible representations of $G(\Q_v)$. Let $v \nmid \infty$, $\tau_v \in {\rm Irr}(\GO(X,\Q_v))$  and $\Pi_v \in {\rm Irr}({\rm GSp}_4(\Q_v))$. We say that $\tau_v$ and $\Pi_v$ correspond (or that $\tau_v$ corresponds to $\Pi_v$) if there is a non-zero $R$-homomorphism from $\omega$ to $\Pi_v \otimes \tau_v$. If $\tau_v$ is tempered and corresponds to $\Pi_v$, then $\Pi_v$ is unique by~\cite[Lemma~4.1]{Takeda2009}. In that case, we denote $\Pi_v$ by $\theta(\tau_v)$. For $v = \infty$, we refer the reader to~\cite[Lemma~4.2]{Takeda2009}
for the appropriate definitions.

\medskip

(b) For all finite places $v=p$ of $\Q$, Roberts \cite[pp.277/278]{Ro01} defines extensions of $\sigma_v$ (certain subrepresentations of ${\rm Ind}_{\GSO}^{\GO} \sigma_v$) to representations $\sigma_v^+$ of $\GO(X,\Q_p)$ such that $\theta(\sigma_v^+)^{\vee}$ is the unique generic representation of ${\rm GSp}_4(\Q_p)$ with $L$-parameter satisfying (\ref{Lparam}). 

At $v=\infty$ we define $\sigma^-_v$ to be the representation denoted by $(\pi_{n,s}, \epsilon, \delta)$ for $n=k-1$, $s=1-k$, $\epsilon=\mathbf{1}$ and $\delta=-1$ in~\cite[p.394]{HST}. 
By \cite[Lemma~12]{HST} (see also \cite[Proposition 6.5(2)]{Takeda2011}) and \cite[Corollary 3]{HST} (which extends to $n=1$),
we know that $\sigma^-_v$  corresponds to a holomorphic limit of discrete series representation $\Pi_{\infty}$ of ${\rm GSp}_4(\bfR)$ with Harish-Chandra parameter $(k-1,0)$ whose $L$-parameter satisfies (\ref{Lparam}).

By \cite[Proposition~5.4]{Takeda2009}, we obtain an irreducible cuspidal automorphic representation of ${\rm GO}(X,\A_{\Q})$ by setting
$$\hat{\sigma}:=\bigotimes_{v<\infty} \sigma^+_v \otimes \sigma^-_{\infty}.$$ 

(c) We now consider the global theta lift $\Theta(V_{\hat{\sigma}})$, which is the space generated by the ${\rm GSp}_4(\A_{\Q})$ automorphic forms $\theta(f; \varphi)$ for all $f \in V_{\hat{\sigma}}$ and $\varphi \in S(X(\A_{\Q})^2)$. (For details of this definition we refer to \cite[Section 5]{Ro01} and \cite[p.11]{Takeda2009}.) 

We know by the local non-vanishing and \cite[Theorem~1.2]{Takeda2009} that $\Theta(V_{\hat{\sigma}})  \neq 0$. Since $\pi \not \cong \pi^c$ we also know that $\Theta(V_{\hat{\sigma}})$ occurs  in the space of cuspforms by Takeda~\cite[Theorem~1.3(2)]{Takeda2009} and \cite[Lemma~5]{HST}. Now let $\Pi$ be an irreducible quotient of $\Theta(V_{\hat{\sigma}})$. Then the representations $\Pi_v$ and $\hat{\sigma}_v^{\vee}$ correspond for all places $v$,  so by~\cite[Lemmas 4.1 and 4.2]{Takeda2009} and \cite[Theorem 1.8]{Ro01} we get $\Pi_v=\theta(\hat{\sigma}^{\vee}_v)=\theta(\hat{\sigma}_v)^{\vee}$ (where the last equality holds by \cite[Proposition~1.10]{Ro01}).

\medskip

(d) The existence of the Siegel paramodular form  is now proved exactly as in \cite{JLR}, but using the argument from \cite[Section~3.1]{SS2011} for defining the vector-valued Siegel modular form $F$. 

The Hecke eigenvalues, epsilon and $L$-factors for the finite part of the automorphic representation $\Pi$ are identical to those in the main theorem of \cite{JLR}. (Note that \cite{JLR} uses the notations $\pi_0$ and $\pi$ instead of our $\pi$ and $\Pi$.) 


 To match the classical spinor $L$-factor of \cite[p.173]{Arakawa83} at unramified places, we see that the shift in the argument of the Euler factors of the local representations $\Pi_v$ given in~\cite[Proposition 4.2]{JLR} to those of $F$ is given by $s \mapsto s - \frac{k-1}{2}$. By the calculations in \cite{JLR} this means that, for all primes $p$,  we have an equality of Euler factors $$L_p(s+\frac{k-1}{2},F)=L_p(s,\pi):=\prod_{\p \mid p} L_{\p}(s+\frac{1}{2}, f),$$ where the $L_p(s,F)$ were defined in Section \ref{sec:siegel-mod-forms} and $L_{\p}(s+\frac{1}{2}, f)$ in Section \ref{sec:bianchi-mod-forms} for the newform $f$ corresponding to $\pi$. As in \cite{JLR}, the functional equation of $\pi$ implies that the completed $L$-function satisfies the functional equation $$\Lambda(k-s,F)=\left(\prod_{p \mid N} \epsilon_p\right) N^{s-k/2} \Lambda(s,F),$$ where $$\Lambda(s,F)=(2 \pi)^{-2s}\Gamma^2(s) L(s,F).$$
\end{proof}

\begin{rem}\rm We end this section with two remarks.
\begin{enumerate}
\item
 Arthur's multiplicity formula and~\cite[Theorem~2.2]{CPMok} 
imply that the multiplicity of $\Pi$ in the space of cuspforms is one. This has been proved in~\cite[Theorem~8.6]{Ro01} in the real quadratic case using the multiplicity preservation of the theta correspondence (\cite[Proposition~5.3]{Ro01}).
\item An alternative construction of the lift of automorphic representations has been proven by P.S. Chan \cite{PSChan} using  trace formulas (under some local conditions on $\pi$). Recently, C. P. Mok  \cite{CPMok} has also described how to obtain this lift from Arthur's endoscopic classification, based on work of Chan and Gan which relates Arthur's local correspondence with that of Gan-Takeda. 
\end{enumerate}
\end{rem}

\section{\bf Application to paramodularity}\label{sec:paramodularity}

In this section, we use our lifting result to prove the following theorem. 
\begin{thm} \label{thm:modular}
Let $C'/ \Q$ be the curve defined by
$$y^2 = 31x^6 + 952x^5 - 5764x^4 - 3750x^3 + 5272x^2 - 7060x + 4783,$$
and $B$ the Jacobian of $C'$. Then, $B$ is a paramodular abelian surface 
of conductor $223^2$ in the sense that it satisfies Conjecture~\ref{conj:BK}.
\end{thm}

We obtain Theorem~\ref{thm:modular} as a consequence of the following statement. 

\begin{thm}\label{thm:surface-egr}
Let $K=\Q(\sqrt{-223})$ and $w = \frac{1+\sqrt{-223}}{2}$, and consider the curve
$$ C : y^2 +Q(x) y= P(x),$$
where
\begin{align*}
P &:= -8x^6 + (54w - 27)x^5 + 9103x^4 + (-14200w + 7100)x^3 - 697185x^2\\
&\qquad\qquad{} + (326468w - 163234)x + 3539399,\\
Q &:= x^3 + (2w - 1)x^2 - x. 
\end{align*}
Let $A = \mathrm{Jac}(C)$ be the Jacobian of $C$. 
Then we have the following:
\begin{enumerate}
\item The curve $C$ is a global minimal model for the base change of $C'$ to $K$ and  
it has everywhere good reduction.  
\item The surface $A$ has real multiplication by $\Z[\sqrt{2}]$, and there exists
a Bianchi newform $f$ of level $(1)$ and weight $2$ and trivial central character such that $f^\sigma = f^\tau$ and 
$$L(A, s) = L(f, s) L(f^\tau, s),$$
where
$\langle \sigma \rangle = \Gal(K/ \Q)$ and $\langle \tau \rangle = \Gal(\Q(\sqrt{2})/ \Q)$.
\end{enumerate} 
\end{thm}

\begin{rem}\label{rem:surface-egr}\rm
Theorem~\ref{thm:surface-egr} (b) implies that $A$ is modular in the sense that
its $L$-series is given by the product of those of the forms $f$ and $f^\tau$. So, this is an
instance of the Eichler-Shimura Conjecture (\cite[Conjecture 3]{taylor_icm}) in dimension 
$2$. We think that this is the first non-trivial such example over an imaginary quadratic field. 
(There is a significant amount of numerical data going back to \cite{grunewald, cremona}
which supports this conjecture in the case of elliptic curves.)
\end{rem}

The rest of this section is dedicated to proving Theorem~\ref{thm:surface-egr}. But first,
we show how to derive Theorem~\ref{thm:modular} from it.

\begin{proof} The equality $f^\sigma = f^\tau$ implies that the form $f$ cannot be a base change. 
Moreover, as the abelian surface $A$ is modular by $f$, the cuspidal irreducible automorphic representation 
associated to $f$ is tempered. So by Theorem~\ref{thm:main} it admits a lift $g$ of weight $2$ to $\GSp_4/\Q$ with paramodular level $223^2$. By construction, the identity $L(A, s) = L(f, s) L(f^\tau, s)$ implies that $L(g,s) = L(B,s)$. 
So $B$ is paramodular. 
\end{proof}



\subsection{\bf The abelian surface} \label{sec:abelian-surface}
As we mentioned in Remark~\ref{rem:surface-egr} above, the surface $A$ satisfies the
Eichler-Shimura Conjecture. It was in fact located via explicit computations of Bianchi modular 
forms. More specifically, we used the extensive data provided in~\cite{rahm-sengun}.

To simplify notations, let $S_2(\mathbf{1}):=S_2(\mathcal{O}_K,\mathbf{1})$ be the space of Bianchi cusp forms of weight $2$, level $\gN = (1)$ and
central character $\chi = \mathbf{1}$. Let $S_2^{bc}(\mathbf{1})$ be the subspace of  $S_2(\mathbf{1})$ which 
consists of twists of those Bianchi cusp forms which arise from classical elliptic newforms via base change and of those 
Bianchi cusp forms which are CM (see \cite{fgt}), and $S_2^{bc}(\mathbf{1})^{\perp}$ its orthogonal complement with respect to the Petersson
inner product. We will call the newforms in $S_2^{bc}(\mathbf{1})^\perp$ {\it genuine}. Given a newform $f$, we 
let $\CO_f = \Z[a_\p(f): \p \subset \CO_K]$  denote the order generated by the Hecke eigenvalues of $f$ and $L$ the field of fractions of $\CO_f$. 

Of all the $186$ imaginary quadratic fields $K$ in~\cite{rahm-sengun} 
(including all those $153$ for which $|D| < 500$), there are {\it only}\footnote{Note that our $S_2(\mathbf{1})$ 
corresponds to $S_0(1)^+$ of \cite{rahm-sengun} by the arguments in~\cite[Sections 2.4 and 2.5]{Lingham}.} 
six for which $S_2^{bc}(\mathbf{1}) \subsetneq S_2(\mathbf{1})$. In each case, $S_2^{bc}(\mathbf{1})^\perp$ is an irreducible Hecke module 
of dimension $2$, except for $|D| = 643$ when there are two newforms, with rational Hecke eigenvalues, that are $\Gal(K/\Q)$-conjugate.
Table~\ref{table:bianchi-newforms} provides a summary of this data. 

\begin{table}[h]
\caption{Genuine Bianchi newforms of weight $2$ over $\Q(\sqrt{D})$ found in \cite{rahm-sengun}.} 
\label{table:bianchi-newforms}
\begin{tabular}{ >{$}c<{$}   >{$}c<{$}   >{$}c<{$}   >{$}c<{$}  >{$}c<{$}  >{$}c<{$} >{$}c< {$}}
\toprule
| D |  & 223 & 415 & 455 & 571 & 643& 1003 \\ 
\midrule
\dim S_2^{bc}(\mathbf{1})^{\perp} & 2 & 2 & 2& 2 & 1+1 & 2\\
L & \Q(\sqrt{2}) & \Q(\sqrt{3}) & \Q(\sqrt{5}) & \Q(\sqrt{5}) & \Q  & \Q(\sqrt{7})\\ 
{}[\CO_L: \CO_f]& 1 & 11 & 2 & 2 & 1 & 1\\
\bottomrule
\end{tabular}
\end{table}

\begin{rem}\label{rem: indices}\rm
  The index entries in Table~\ref{table:bianchi-newforms} are based on
  finite sets of Hecke eigenvalues.  Since there is no analogue of
  Sturm's bound for Bianchi modular forms, the last row entries are
  not proven to be correct. However, we strongly expect them to
  reflect the truth.
\end{rem}

\medskip 
Let $f$ be any of the genuine
Bianchi newforms listed in Table~\ref{table:bianchi-newforms}.
Let $\sigma$ and $\tau$ denote the non-trivial elements of
$\Gal(K/\Q)$ and $\Gal(L/\Q)$ respectively.  
Let $f^\tau$ be $\Gal(L/\Q)$-conjugate of $f$, which is 
determined by the relation
$$a_\p(f^\tau) = \tau(a_\p(f))\,\,\text{for all primes}\,\, \p,$$
and $f^\sigma$ be the $\Gal(K/\Q)$-conjugate of $f$ which is
determined by
$$a_\p(f^\sigma) = a_{\sigma(\p)}(f)\,\,\text{for all primes}\,\, \p.$$

Except for the discriminant $|D| = 643$, dimension considerations show that $f^\tau = f^\sigma,$ so we have
\begin{equation*}
a_{\sigma(\p)}(f)=\tau( a_{\p}(f)) \,\,\text{for all primes}\,\, \p.
\end{equation*}
A refinement of the Eichler-Shimura Conjecture implies that there exists an abelian surface $A/K$ with
$\CO_f \subset \End_K(A)$ ({\it i.e.} $A$ has real multiplication by $\CO_f$) 
such that $$L(A, s) = L(f, s) L(f^\tau, s).$$

There are only two pairs $(|D|, \mathrm{Disc}(L))$, where provably the
Hecke ei\-gen\-va\-lues of the newform $f$ generate the ring of integers of $L$,
namely $(223, 8)$ and $(1003,28)$. For the first pair, the (conjectured) abelian 
surface attached to the form is principally polarisable (see~\cite[Corollary~2.12 and
Proposition~3.11]{gonzalez-guardia-rotger05}) in contrast to the second pair, for 
which this does not seem to be the case. So we will only focus on the first pair, 
for which we found the associated abelian surface. In that case, we have $\CO_f = \Z[\sqrt{2}]$ 
(see Table~\ref{table:eigenvalues-iqsqrt223} for  the Hecke eigenvalues of the form, which we 
computed using Yasaki's algorithm~\cite{yasaki} implemented in \textsf{Magma}~\cite{magma97}). 
Our current approach does not allow us to find the remaining
surfaces. (We elaborate on this in Remark~\ref{rem:other-surfaces}.) 

\begin{table}[h]
\caption{Arithmetic data associated to the genuine Bianchi newform $f$ of weight $2$ and 
level $(1)$ over $\Q(\sqrt{-223})$. Here, $w =\frac{1+\sqrt{-223}}{2}$ and $e = \sqrt{2}$.}
\label{table:eigenvalues-iqsqrt223}\small
\begin{tabular}{ >{$}r<{$}   >{$}r<{$}   >{$}r<{$}   >{$}r<{$}  }
\toprule
\mathrm{N}\p & \p & a_{\p}(f) & (x^2 - a_{\p}(f) x + \mathrm{N}(\p))(x^2 - \tau(a_{\p}(f)) x + \mathrm{N}(\p))\\
\midrule
2 & [ 2, w + 1 ] & e - 1 & x^4 + 2x^3 + 3x^2 + 4x + 4 \\
2 & [ 2, w + 2 ] & -e - 1 & x^4 + 2x^3 + 3x^2 + 4x + 4 \\
7 & [ 7, 2w + 5 ] & -e + 2 & x^4 - 4x^3 + 16x^2 - 28x + 49 \\
7 & [ 7, 2w ] & e + 2 & x^4 - 4x^3 + 16x^2 - 28x + 49 \\
9 & [ 3 ] & -3 & x^4 + 6x^3 + 27x^2 + 54x + 81 \\ 
17 & [ 17, 2w + 9 ] & -2e - 1 & x^4 + 2x^3 + 27x^2 + 34x + 289 \\
17 & [ 17, 2w + 6 ] & 2e - 1 & x^4 + 2x^3 + 27x^2 + 34x + 289 \\
19 & [ 19, 2w + 8 ] & e - 4 & x^4 + 8x^3 + 52x^2 + 152x + 361 \\
19 & [ 19, 2w + 9 ] & -e - 4 & x^4 + 8x^3 + 52x^2 + 152x + 361 \\
25 & [ 5 ] & 0 & x^4 + 50x^2 + 625 \\
29 & [ 29, 2w + 25 ] & 2e + 3 & x^4 - 6x^3 + 59x^2 - 174x + 841 \\
29 & [ 29, 2w + 2 ] & -2e + 3 & x^4 - 6x^3 + 59x^2 - 174x + 841 \\
31 & [ 31, 2w + 4 ] & 4e - 2 & x^4 + 4x^3 + 34x^2 + 124x + 961 \\
31 & [ 31, 2w + 25 ] & -4e - 2 & x^4 + 4x^3 + 34x^2 + 124x + 961 \\
\bottomrule
\end{tabular}
\end{table}


For the discriminant $|D|=223$, since $A$ is principally polarised and has real multiplication by $\CO_L = \Z[\sqrt{2}]$, it corresponds
to a $K$-rational point on the Humbert surface $Y_{-}(8)$ of discriminant $8$. 
In their recent paper~\cite{ek}, Elkies and Kumar give an explicit rational model for $Y_{-}(8)$ as a 
double-cover of the weighted projective space $\mathbf{P}^2_{r,s}$. We look for $A$ using this 
model. In fact, the same heuristic as in~\cite[Proposition 6 and Remark 5]{dk} show that $A$ must be 
the base change of a surface $B$ defined over $\Q$. Indeed, our newform satisfies the identity 
$f^\tau = f^\sigma$. So by Theorem~\ref{thm:main}, it admits a lift to a classical Siegel newform $g$
of genus $2$, weight $2$ and level $D^2$, with integer coefficients. Moreover, $g$ is not a Gritsenko lift. So,
assuming Conjecture~\ref{conj:BK}, $g$ corresponds to 
an abelian surface $B$ over $\Q$ such that $\End_{\Q}(B) = \Z$ and $L(B,s) = L(g, s)$. It follows that
$A = B \otimes_\Q K$. 

We recall that, as a double-cover of $\mathbf{P}^2_{r,s}$, $Y_{-}(8)$ is given by the
equation
\[
z^2 = 2(16rs^2+32r^2s-40rs-s+16r^3+24r^2+12r+2).
\] 
The Igusa-Clebsch invariants as a point in $\Proj^3_{(1:2:3:5)}$ are given by
\[(I_2: I_4: I_6: I_{10}) = 
\left( -\frac{24B_1}{A_1} : -12A : \frac{96AB_1-36A_1B}{A_1} : -4A_1B_2 \right),
\]
where
\begin{align*}
A_1 &= 2rs^2, \\
A &= -(9rs+4r^2+4r+1)/3, \\
B_1 &= (rs^2(3s+8r-2))/3, \\
B &= -(54r^2s+81rs-16r^3-24r^2-12r-2)/27, \\
B_2 &=  r^2.
\end{align*}

The surface $B$ corresponds to the parameters $s=-2$ and $r = 8$. It is the Jacobian of the curve
$$C':\, y^2 = 31x^6 + 952x^5 - 5764x^4 - 3750x^3 + 5272x^2 - 7060x + 4783,$$
whose discriminant is $-2^{20}223^{15}$. The curve $C'\otimes_\Q K$ admits the 
global minimal model $C$ listed in Theorem~\ref{thm:surface-egr}.
One verifies that the curve $C$ has discriminant $1$ and integer Igusa-Clebsch invariants 
$$I_2 = -24,\,\, I_4 = -540,\,\, I_6 = 4968,\,\, I_{10} =4096.$$  
So $A = \mathrm{Jac}(C)$ has everywhere good reduction.
 
\begin{rem}\label{rem:other-surfaces}\rm
  For the discriminants $|D| = 415, 1003$, the class number of $L$ is
  $2$. The same heuristics as in~\cite[Remark 8]{dk} suggest that
  the surface $A$ is likely not principally polarised. Further, for
  $|D| = 415, 455, 571$, $A$ does not have real multiplication by the
  maximal order in $L$.  So finding these surfaces will require
  working with more general Humbert surfaces for which no explicit
  models are yet available.
\end{rem}

\subsection{\bf Proof of modularity}\label{sec:modularity}
The above discussion already proves Theorem~\ref{thm:surface-egr} (a).
We will now show that $A$ is modular hence completing the proof of the theorem. 
In fact, we already have strong evidence that this is the case. Indeed, let $\p$ be a prime,
and $L_\p(A, s)$ (resp. $L_\p(f, s)$ and $L_\p(f^\tau, s)$) the Euler factor of $A$ (resp. $f$ and $f^\tau$) at $\p$. 
As a built-in of our search method, we already know that for each prime $\p$ listed in Table~\ref{table:eigenvalues-iqsqrt223},  
we have 
$$L_\p (A, s) = L_\p(f, s) L_\p(f^\tau, s) = Q_\p(\mathrm{N}\p^{-s})^{-1},$$ where
$Q_\p(x) := (1 - a_{\p}(f) x + \mathrm{N}(\p) x^2)(1 - \tau(a_{\p}(f)) x + \mathrm{N}(\p) x^2).$

\medskip
Let $\lambda \subset \CO_L$ be a prime ideal. We recall that the $\lambda$-adic Tate module of $A$ is
given by 
$$T_\lambda(A) := \varprojlim_{n} A[\lambda^n] \simeq \CO_{L,\lambda} \times \CO_{L,\lambda},$$ where
$\CO_{L,\lambda}$ is the completion of $\CO_L$ at $\lambda$, and 
$$A[\lambda^n] = \left\{x \in A(\Qbar): \alpha x = 0\,\,\forall \alpha \in \lambda^n\right\}.$$
This is naturally endowed with an action of $\Gal(\Qbar/K)$ giving rise to the $\lambda$-adic representation
$$\rho_{A, \lambda}: \Gal(\Qbar/K) \rightarrow \GL_2(\CO_{L,\lambda}).$$

In the rest of this section, we will drop the reference to the prime in our notations as we are only
interested in the prime $\lambda = \lambda_2$ (above $2$) for which $\CO_{L,\lambda_2} =\Z_2[\sqrt{2}]$. 
So our aim is to show that the $\lambda_2$-adic Tate module
$$\rho_A: \Gal(\Qbar/K) \rightarrow \GL_2(\Z_2[\sqrt{2}])$$
is isomorphic to the representation
$$\rho_f: \Gal(\Qbar/K) \rightarrow \GL_2(\overline{\Z}_2)$$
associated to $f$ by work of Taylor et al \cite{HST, taylor_94,  berger-harcos, CPMok}.
But, in order to do so, we must first determine the coefficient field of $\rho_f$.
\begin{lem}\label{lem:coefficient-field}
The coefficient field of $\rho_f$ is $\Q_2(\sqrt{2})$
so we have $$\rho_f: \Gal(\Qbar/K) \rightarrow \GL_2(\Z_2[\sqrt{2}]).$$
\end{lem}

\begin{proof}
By construction, the image of $\rho_f$ lies in an extension of $\Q_2$
of degree at most $4$. The prime $2$ is split in $K$, and the eigenvalues 
of the Frobenii at the primes above it are distinct and do not add up to zero
(see Table~\ref{table:eigenvalues-iqsqrt223}), so by \cite[Corollary 1]{taylor_94} 
we can take the coefficient field $L_f$ to be the one given by the polynomial 
$x^4+2x^3+3x^2+4x+4$. But, there are two split primes above $2$ in $L_f$, and the
completion of $L_f$ at either of them is isomorphic to $\Q_2(\sqrt{2})$. So the
image of $\rho_f$ is in fact contained in $\GL_2(\Z_2[\sqrt{2}])$.
\end{proof}

From Lemma~\ref{lem:coefficient-field}, we now have
$$\rho_A, \rho_f: \Gal(\Qbar/K) \rightarrow \GL_2(\Z_2[\sqrt{2}]).$$
We denote their reductions modulo $2$ by
$\bar{\rho}_A$ and $\bar{\rho}_f$ respectively. We will show that
$\rho_A \simeq \rho_f$ by making use of the following version
of the so-called Faltings-Serre criterion~\cite{serre, DGP}.

\begin{thm}[Faltings-Serre]\label{thm:faltings-serre} 
Let
\[
\rho_1,\rho_2: \Gal(\Qbar/K) \rightarrow \GL_2(\overline{\Z}_2) 
\]
be two continuous representations, whose reductions modulo $2$ are $\bar{\rho}_1$ 
and $\bar{\rho}_2$. Suppose that 

\begin{itemize}
\item[(i)] $\det (\rho_1) = \det (\rho_2)$;
\item[(ii)] $\rho_1$ and $\rho_2$ are unramified outside a finite set of primes $S$;
\item[(iii)] $\bar{\rho}_1$ and $\bar{\rho}_2$ are absolutely irreducible 
and isomorphic.
\end{itemize}
Then there exists a computable {\rm finite} set of primes $T$ such that 
$\rho_1 \simeq \rho_2$ if and only if 
 $$\mathrm{Tr}(\rho_1(\mathrm{Frob}_{\p})) = \mathrm{Tr}(\rho_2(\mathrm{Frob}_{\p}))$$
for all $\p \in T$.
\end{thm}
There is an explicit description of the set $T$, which we recall here
for the sake of completeness (see~\cite{DGP} for more details on
this). Assume that the image lies in $\GL_2(\F_2)$, which will be the
case in our example.


Let $M$ be the fixed field of $\im(\overline{\rho}_1)$, the residual
image of $\rho_1$.  (We have the same fixed field for
$\im(\overline{\rho}_2)$ since $\bar{\rho}_1$ and $\bar{\rho}_2$ are
isomorphic.) Let $M_2^0(\F_{2})$ be the set of all trace zero $2 \times 2$ matrices with coefficients in $\F_2$.
We note that this is a $2$-group of order $8$, and we consider the set of all extensions $\widetilde M$ of
$M$ that are unramified outside $S$, such that $\widetilde{M}$ is Galois over $K$ and 
$\Gal(\widetilde{M}/K)\simeq M_2^0(\F_{2})\rtimes \im(\overline{\rho}_1)$. Each such $\widetilde{M}$ is a compositum of 
quadratic extensions of $M$, so there is a canonical isomorphism
$\varphi_{\widetilde{M}}:\Gal(\widetilde{M}/K)\simeq M_2^0(\F_{2})\rtimes \im(\overline{\rho}_1)$. 
For algorithmic purpose, this set is determined explicitily using class field theory (see~\cite[Lemma 5.6]{DGP}). For
each $\widetilde{M}$, we then find a prime ideal $\p_{\widetilde{M}}\subset \CO_K$ such that 
$\varphi_{\widetilde{M}}(\Frob_{\p_{\widetilde{M}}})=(A,B)$ with $\trace(AB) \neq 0$. 
The set $\{\p_{\widetilde M}\}$ thus obtained has the desired properties.

\begin{proof}[Proof of Theorem~\ref{thm:surface-egr} (b)] 
We will use Theorem~\ref{thm:faltings-serre} to show that $\rho_A \simeq \rho_f$; 
hence that $A$ is modular. 

We note that $\rho_A$ and $\rho_f$ are unramified away from $\lambda_2$ (for $\rho_f$ this uses the 
local-global compatibility results in~\cite{CPMok}) so they satisfy (ii) with
$S = \{\p \mid 2 \}$. Condition (i) is satisfied by $\rho_f$ since $f$ has trivial Nebentypus. It
is satisfied by $\rho_A$ by basic properties of Tate modules. So we only need to check 
(iii), find the set $T$ and show that 
\begin{align*}
\trace(\rho_A(\Frob_{\p})) = \trace(\rho_f(\Frob_{\p}))\,\,\text{for all}\,\,\p \in T.
\end{align*}
Since $\CO_L/\lambda_2 = \F_2$, $\im(\overline{\rho}_f)$ and $\im(\overline{\rho}_A)$ are both contained in 
$\GL_2(\F_2) \simeq S_3$, where $S_3$ is the symmetric group on $3$ elements. 

\smallskip
\noindent {\bf Claim:} $\im(\overline{\rho}_A) =
\im(\overline{\rho}_f) = \GL_2(\F_2)$. In particular,
$\bar{\rho}_A$ and $\bar{\rho}_f$ are absolutely irreducible.

\smallskip
First, recall that $\trace(\bar{\rho}_f(\Frob_{\p})) = a_{\p}(f)
\bmod \lambda_2$, and that since $\rho_f$ is unramified away from $2$,
$\bar{\rho}_f(\Frob_{\p})$ is either trivial or has order $3$ for every odd prime $\p$. 
In particular, for $\p =(3)$, $\trace(\bar{\rho}_f(\Frob_\p)) = -3 = 1 \in \F_2$ implies that 
$\bar{\rho}_f(\Frob_\p)$ has order $3$. Next, for $\p \mid 2$, the representation
$\rho_f$ is ordinary at $\p$, so the restriction of $\bar{\rho}_f$ to the 
decomposition group $D_\p$ is of the form
$$\bar{\rho}_f|_{D_\p}\sim \begin{pmatrix}1&\ast \\ 0 &1 \end{pmatrix} \mod \lambda_2.$$
If $\mathrm{Im}(\bar{\rho}_f|_{D_\p})$ was trivial for both primes $\p \mid 2$, then
$\bar{\rho}_f$ would be unramified at $2$ (and hence everywhere). In that case, the
fixed field of $\ker (\bar{\rho}_f)$ would be an unramified cubic extension of $K$, 
which is impossible since the class number of $K$ is $7$. 
So  $\im(\overline{\rho}_f)$ contains an element of order $2$, hence 
$\im(\overline{\rho}_f) = \GL_2(\F_2)$. Similarly, one shows that 
$\im(\overline{\rho}_A) = \GL_2(\F_2)$.

\medskip
Next, let $M_f$ and $M_A$ be the Galois extensions of $K$ cut out by
$\overline{\rho}_f$ and $\overline{\rho}_{A}$ respectively. Then $M_A$
and $M_f$ are $S_3$-extensions of $K$ ramified at $\lambda_2$ only. To
check (iii), we will show that there is a unique such extension. This
will force the isomorphism $M_f \simeq M_{A}$.

Let $N_A$ and $N_f$ be the respective normal closures of $M_A$ and
$M_f$ over $\Q$.  Since $f^\sigma = f \bmod \lambda_2$, $\bar{\rho}_f$ is a 
base change, so $\Gal(N_f/\Q) \simeq \Z/2\Z \times S_3$.
Similarly, we have $\Gal(N_A/\Q) \simeq \Z/2\Z \times S_3$. So each
extension comes from an $S_3$-extension of $\Q$ which is unramified
outside $\{2, 223\}$. Now we will show that there is a {\em unique}
such $S_3$-extension $N / \Q$.

First, we observe that there are exactly $7$ possible quadratic subfields, given by $\Q(\sqrt{d})$ with 
$d=-1,2,-2,223,-223,2\cdot223,-2\cdot223$. From Table \ref{table:eigenvalues-iqsqrt223}, 
we know that $\trace(\bar{\rho}_f(\Frob_\p))$ is odd for the primes above $3,17,29$. The prime $3$ is inert for
$d=-1,2,-223$ and $2\cdot223$. The prime $29$ is also inert for $d=-2$ and
$-2\cdot223$. So we deduce that the only possible quadratic subfield is
$\Q(\sqrt{223})$. 

We now turn to the cubic extension $N/ \Q(\sqrt{223})$. 
Since our representation $\overline{\rho}_f$ is unramified at $223$, this
extension cannot be ramified at the prime above $223$, so it is only ramified at $\lambda_2$. 
To see this, consider the following diagram
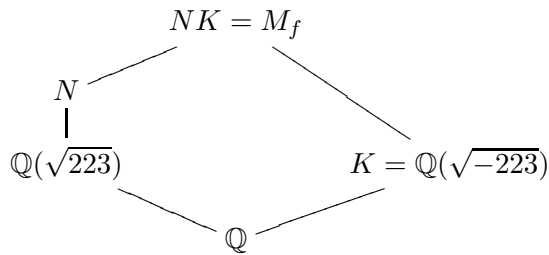
\[
\xymatrix@=1em{
  & N K = M_f \ar@{-}[ddr]\\
N \ar@{-}[ur]\ar@{-}[d] \\
\Q(\sqrt{223})\ar@{-}[dr] & & K=\Q(\sqrt{-223})\ar@{-}[dl]\\
& \Q & 
}
\]
and note that on the right the ramification index of $223$ is $2$. So
the same holds on the left. Again using class field theory, we see
that there is a unique such extension $N/ \Q$ given by the polynomial
$$x^6 - 2 x^5 - 29 x^4 + 90 x^3 - 58 x^2 -8 x +8.$$
So $\bar{\rho}_A \simeq \bar{\rho}_f$ and $\{\rho_A, \rho_f\}$ satisfy
all three hypotheses of Theorem~\ref{thm:faltings-serre}.

We conclude that $\rho_f$ and $\rho_A$ are isomorphic using the same recipe as in~\cite{DGP}. 
For this, we need to compare the traces of the Frobenii at the primes in $\CO_K$ that lie
above the rational primes in $\{ 3, 5, 7, 19, 29, 31 \}$. Since we already know that the traces
for these particular primes match, we are done.
\end{proof}


\bibliographystyle{amsalpha}
\bibliography{biblio}
\end{document}